\documentclass[twoside, 12pt]{article}
\usepackage{mathrsfs}
\usepackage{amssymb, amsmath, mathrsfs, amsthm}
\usepackage{array}
\usepackage{algorithm}
\usepackage{algpseudocode}
\usepackage{float}
\usepackage{graphicx}
\usepackage{amsmath}  
\allowdisplaybreaks[4]  
\usepackage{color}
\usepackage{enumitem}
\setlist[enumerate,1]{label=(\arabic*)}
\usepackage{tikz}
\usepackage[top=2cm, bottom=2cm, left=2cm, right=2cm]{geometry}
\usepackage{float, caption, subcaption}
    \usepackage{diagbox}
\input{epsf}
\usepackage[colorlinks,
  linkcolor=blue,
 anchorcolor=blue,
 citecolor=blue]{hyperref}

\newtheorem{theorem}{Theorem}[section]
\newtheorem{proposition}[theorem]{Proposition}
\newtheorem{lemma}[theorem]{Lemma}

\theoremstyle{definition}

\theoremstyle{remark}

\numberwithin{equation}{section}

\newcommand{\ZZ}{\mathbb Z}

\newcommand{\calE}{\mathcal E}
\newcommand{\set}[1]{\left\{#1\right\}}
\newcommand{\abs}[1]{\left\lvert #1\right\rvert}
\newcommand{\floor}[1]{\left\lfloor #1\right\rfloor}
\newcommand{\ceil}[1]{\left\lceil #1\right\rceil}

\begin{document}
\title{\textbf{On the minimum number of monochromatic solutions to the strict Schur inequality in 2-colored integer intervals with negative left endpoint} \footnote{Supported by the National Science Foundation of China
(Nos. 12471329 and 12061059).} }

\author{Gang Yang\footnote{Graduate School of Environment and Information Sciences, Yokohama
National University, 79-2 Tokiwadai, Hodogaya-ku, Yokohama 240-8501,
Japan. {\tt gangyang98@outlook.com}}, \ Jinxia Liang \footnote{
        School of Mathematics and Statistics, Qinghai Normal University,
        Xining, Qinghai 810008, China.
        {\tt ljxqhsd@aliyun.com} }, \  Yaping Mao \footnote{Academy of Plateau
Science and Sustainability, and School of Mathematics and Statistics, Qinghai Normal University, Xining, Qinghai 810008, China. {\tt yapingmao@outlook.com; myp@qhnu.edu.cn}},  \ 
Chenxu Yang\footnote{Corresponding author: School of Science, Tianjin Chengjian University, Tianjin, Tianjin 300384, China.
{\tt cxyang@tcu.edu.cn}}, \  Ayun Zhang \footnote{School of Mathematics and Statistics, Qinghai Normal University, Xining, Qinghai 810008,
China. {\tt zhangay@qhu.edu.cn}}
}
\date{}
\maketitle

\begin{abstract}
Kosek, Robertson, Sabo, and Schaal studied the minimum number \(M_k(n)\) of monochromatic solutions to the strict Schur inequality system
$x_1\le x_2\le x_3$ and $x_1+x_2<x_3$
in \(2\)-colorings of \([k+1,k+n]\). They proved that for every fixed \(k\ge 0\),
$M_k(n)= \frac{n^3}{12(1+2\sqrt2)^2}(1+o_k(1)),$
and left open the case \(k\le -2\). In this paper, we resolve that remaining range. 
\\[2mm]
\end{abstract}

\section{Introduction}

Ramsey theory, named after the British mathematician and philosopher Frank P. Ramsey, is a branch of mathematics that focuses on the appearance of order in a substructure given a structure of a known size; see \cite{GRS90}.
Ramsey theory consists of four core problems: Ramsey theorems establish existence; Ramsey numbers find the minimal order guaranteeing a monochromatic substructure; Ramsey multiplicity compute the minimum number of monochromatic substructures; Ramsey realizations identify the colorings achieving this minimum. the appearance of order in a substructure given a structure of a known size; see \cite{Mao24}. Integer Ramsey theory is the study of monochromatic solutions to equations and patterns in finitely colored integer intervals; see the book \cite{LR04} and the papers on Ramsey multiplicity \cite{BCG10,Datskovsky,PRS,PS26,RZ,Schoen,Th09,TW17}.

Let $\ZZ$ denote the set of all integers. 
For integers $a\le b$, write
$[a,b]=\set{m\in\ZZ:a\le m\le b}$.
We study monochromatic solutions to the strict Schur inequality system
$ \calE: x_1\le x_2\le x_3, x_1+x_2<x_3$.
If $\Delta:[a,b]\to\{0,1\}$ is a $2$-coloring of $[a,b]$, let $s(\Delta;a,b)$ denote the number of monochromatic triples $(x_1,x_2,x_3)$ satisfying $\calE$, where $x_1,x_2,x_3\in [a,b]$. For integers $k\in\ZZ$ and $n\ge 1$, define
\[
M_k(n)=\min_{\Delta:[k+1,k+n]\to\{0,1\}}s(\Delta;k+1,k+n).
\]
Thus $M_k(n)$ is the multiplicity of $\calE$ on the interval $[k+1,k+n]$.

Multiplicity problems of this type are classical in additive Ramsey theory. For the Schur equation $x+y=z$, the asymptotic minimum number of monochromatic solutions was determined independently by Robertson and Zeilberger and by Schoen, and later reproved by Datskovsky \cite{Datskovsky,KRSS,RZ,Schoen}. For the inequality $\calE$, Kosek, Robertson, Sabo, and Schaal proved that for every fixed $k\ge 0$ the minimum is asymptotically $\frac{n^3}{12(1+2\sqrt2)^2}$,
and, more importantly for the present paper, they proved that on a positive interval the extremal coloring with a fixed minority-class size is a two-block coloring with the minority color placed first \cite[Theorems~1 and~7]{KRSS}. At the end of their paper they observed that the range $k\le -2$ requires additional analysis and left it open \cite[p.~1134]{KRSS}.

The aim of this paper is to settle that remaining range. Throughout, we write $k=-t, t\in\ZZ$, and $t\ge 2$.
Then $[k+1,k+n]=[1-t,n-t]$.
The interval now contains a fixed nonpositive block and, when $n\ge t+1$, a positive tail. This makes the exact counting problem more delicate than in the purely positive setting.

Our first theorem gives an exact finite reduction.
\begin{theorem}\label{thm:intro-exact}
Let $t\ge 2$ and $k=-t$.
\begin{enumerate}[label=\textnormal{(\alph*)},leftmargin=2.2em]
\item If $1\le n\le t$, then
\[
M_{-t}(n)=
\begin{cases}
\dbinom{\floor{n/2}+2}{3}+\dbinom{\ceil{n/2}+2}{3}, & 1\le n<t,\\[2mm]
\dbinom{\floor{t/2}+2}{3}+\dbinom{\ceil{t/2}+2}{3}-1, & n=t.
\end{cases}
\]

\item If $n\ge t+1$ and $N=n-t$, then
\[
M_{-t}(n)=
\min_{\varepsilon\in\{0,1\}}
\ \min_{0\le a\le t-1}
\ \min_{0\le q\le \floor{N/2}}
G_{t,N,a,\varepsilon}(q),
\]
where
\begin{align*}
G_{t,N,a,\varepsilon}(q)
={}&
S(1,q)+S(q+1,N) \\
&+\binom{a+\varepsilon+1}{2}\,q
+\binom{t-a-\varepsilon+1}{2}\,(N-q) \\
&+a\binom{q+1}{2}
+(t-1-a)\binom{N-q+1}{2} \\
&+\varepsilon\binom{q}{2}
+(1-\varepsilon)\binom{N-q}{2} \\
&+\binom{a+\varepsilon+2}{3}
+\binom{t-a-\varepsilon+2}{3}-1,
\end{align*}
and $S(a,b)$ denotes the exact number of solutions to $\calE$ in the positive interval $[a,b]$.
\end{enumerate}
\end{theorem}
Proposition \ref{prop:Sab} shows that
\[
S(a,b)=
\begin{cases}
0,&b-2a\le 0,\\[2mm]
\dfrac{u(u+1)(4u+5)}{6},&b-2a=2u\ge 2,\\[3mm]
\dfrac{(u+1)(u+2)(4u+3)}{6},&b-2a=2u+1\ge 1.
\end{cases}
\]

In Theorem \ref{thm:intro-exact} \textnormal{(b)}, the parameters have the following meaning. The integer $a$ counts the negative integers of color $0$ in $[1-t,-1]$, the parameter $\varepsilon$ records the color of $0$, and the integer $q$ counts the positive integers of color $0$ in $[1,n-t]$. For each fixed triple $(a,\varepsilon,q)$, the positive part of an extremal coloring is exactly the two-block coloring in \cite[Theorems~1 and~7]{KRSS}.

\begin{theorem}[\cite{KRSS}, Theorem 1]\label{thm1:kosek}
Let \(k\ge 0\), \(m\ge 0\), and \(n\ge 1\) be integers, with \(m\le \frac n2\). If $\Delta:[k+1,k+n]\to\{0,1\}$ is a $2$-coloring of $[k+1,k+n]$ with $|\Delta^{-1}(0)|=m$,
then
\[
s(\Delta;k+1,k+n)\ge S(k+1,k+m)+S(k+m+1,k+n).
\]
\end{theorem}

By the symmetry between the two colors, it follows that
\[
M_k(n)=\min_{0\le m\le n/2}\Bigl(S(k+1,k+m)+S(k+m+1,k+n)\Bigr).
\]
Kosek, Robertson, Sabo, and Schaal then analyzed this minimum asymptotically and obtained the following result.

\begin{theorem}[\cite{KRSS} Theorem 7]\label{thm7:KRSS}
For any given integer \(k\ge 0\), the minimum number of monochromatic solutions to \(x_1+x_2<x_3\) that can occur in any \(2\)-coloring of \([k+1,k+n]\) is
\[
M_k(n)=Cn^3\bigl(1+o_k(1)\bigr),
\]
where
\[
C=\frac{1}{12(1+2\sqrt2)^2}\approx 0.005685622025.
\]
\end{theorem}

Our second theorem is the asymptotic consequence.

\begin{theorem}\label{thm:intro-asymptotic}
For every fixed integer $k\le -2$,
\[
M_k(n)=\frac{1}{12(1+2\sqrt2)^2}n^3+O_k(n^2)
\qquad (n\to\infty).
\]
Equivalently, if $k=-t$ with $t\ge 2$ and $N=n-t$, then
\[
M_{-t}(n)=\frac{1}{12(1+2\sqrt2)^2}N^3+O_t(N^2).
\]
\end{theorem}

The paper is organized as follows. Section~\ref{sec:positive} records the exact solution count on a positive interval. Section~\ref{sec:nonpositive} treats the purely nonpositive case $1\le n\le t$ and proves Theorem~\ref{thm:intro-exact} (a). Section~\ref{sec:exact-reduction} treats the case $n\ge t+1$ and proves Theorem~\ref{thm:intro-exact} (b). Section~\ref{sec:asymptotic} derives the cubic asymptotic formula in Theorem~\ref{thm:intro-asymptotic}. We use the standard convention that the constants implicit in $O_t(\cdot)$ may depend on the fixed integer $t$, and similarly for $O_k(\cdot)$.

\section{Positive-interval preliminaries}\label{sec:positive}

For integers $1\le a\le b$, let
\[
S(a,b)=\left|\set{(x_1,x_2,x_3): x_1\le x_2\le x_3,\ x_1+x_2<x_3, x_1,x_2,x_3\in[a,b]}\right|.
\]
If $a>b$, we set $S(a,b)=0$ by convention. Thus $S(a,b)$ counts all solutions to $\calE$ in the interval $[a,b]$, without reference to any coloring.

We now show the exact value of $S(a,b)$.

\begin{proposition}\label{prop:Sab}
Let $1\le a\le b$. Then
\begin{equation}\label{eq:Sab-sum}
S(a,b)=\sum_{i=a}^{\floor{(b-1)/2}}\binom{b-2i+1}{2}.
\end{equation}
Equivalently,
\[
S(a,b)=
\begin{cases}
0,&b-2a\le 0,\\[2mm]
\dfrac{u(u+1)(4u+5)}{6},&b-2a=2u\ge 2,\\[3mm]
\dfrac{(u+1)(u+2)(4u+3)}{6},&b-2a=2u+1\ge 1.
\end{cases}
\]
\end{proposition}

\begin{proof}
Fix $x_1=i\in[a,b]$. Since $x_1\le x_2\le x_3$ and $x_1+x_2<x_3\le b$, we must have
$i\le x_2\le b-i-1$.
Hence $i\le \floor{(b-1)/2}$ is necessary. Conversely, for each such $i$ and each
$j\in[i,b-i-1]$,
the variable $x_3$ must satisfy $x_3>i+j$ and $x_3\le b$, so the possible values of $x_3$ are exactly
$i+j+1,i+j+2,\dots,b$,
which gives precisely $b-(i+j)$ choices.

Therefore, for fixed $i$, the number of pairs $(x_2,x_3)$ is
\[
\sum_{j=i}^{b-i-1}(b-i-j)=1+2+\cdots +(b-2i)=\binom{b-2i+1}{2}.
\]
Summing over all admissible $i$ yields \eqref{eq:Sab-sum}.

To obtain the closed form, write $i=a+r$. Then
$b-2i=d-2r$.
If $b-2a\le 0$, then there are no admissible values of $i$, so $S(a,b)=0$.

Assume $b-2a\ge 1$. If $b-2a=2u$, then $r=0,1,\dots,u-1$, and therefore
\[
S(a,b)=\sum_{r=0}^{u-1}\binom{2u-2r+1}{2}=\sum_{j=1}^{u}\binom{2j+1}{2}=\frac{u(u+1)(4u+5)}{6}.
\]
If $b-2a=2u+1$, then $r=0,1,\dots,u$, so
\[
S(a,b)=\sum_{r=0}^{u}\binom{2u-2r+2}{2}=\sum_{j=0}^{u}\binom{2j+2}{2}=\frac{(u+1)(u+2)(4u+3)}{6}.
\]
This completes the proof.
\end{proof}

\section{The purely nonpositive case}\label{sec:nonpositive}

We first treat the range $1\le n\le t$, where the interval $[1-t,n-t]$ contains no positive integer.

\begin{theorem}\label{thm:nonpositive}
Let $t\ge 2$ and $1\le n\le t$. Then
\[
M_{-t}(n)=
\begin{cases}
\dbinom{\floor{n/2}+2}{3}+\dbinom{\ceil{n/2}+2}{3}, & 1\le n<t,\\[2mm]
\dbinom{\floor{t/2}+2}{3}+\dbinom{\ceil{t/2}+2}{3}-1, & n=t.
\end{cases}
\]
\end{theorem}

\begin{proof}
Let $J=[1-t,n-t]$.
Because $n\le t$, every element of $J$ is nonpositive.

Consider a triple $(x_1,x_2,x_3)$ with
$x_1\le x_2\le x_3$ and $x_1,x_2,x_3\in J$.
If $x_1<0$, then
$x_1+x_2\le x_2-1<x_2\le x_3$,
so $(x_1,x_2,x_3)$ automatically satisfies $\calE$. If $x_1=0$, then necessarily $n=t$ and
$0=x_1\le x_2\le x_3\le 0$,
which forces $(x_1,x_2,x_3)=(0,0,0)$. This triple is not a solution to $\calE$, because $0+0<0$ is false.

Thus every nondecreasing triple in $J$ is a solution, except for the single triple $(0,0,0)$ when $n=t$.
Now fix a $2$-coloring $\Delta:J\to\{0,1\}$ and let
$m=\abs{\Delta^{-1}(0)}$.
Then color $1$ appears exactly $n-m$ times. The number of nondecreasing monochromatic triples chosen from a color class of size $m$ is the number of multisets of size $3$ drawn from $m$ ordered elements, namely $\binom{m+2}{3}$. Therefore
\[
s(\Delta;1-t,n-t)=\binom{m+2}{3}+\binom{n-m+2}{3}-\delta_{n,t},
\]
where $\delta_{n,t}=1$ if $n=t$ and $\delta_{n,t}=0$ otherwise. Consequently,
\[
M_{-t}(n)=\min_{0\le m\le n}\left(\binom{m+2}{3}+\binom{n-m+2}{3}\right)-\delta_{n,t}.
\]
It remains to minimize the symmetric function
\[
g_n(m)=\binom{m+2}{3}+\binom{n-m+2}{3}.
\]
A direct calculation gives
\[
g_n(m+1)-g_n(m)=\binom{m+2}{2}-\binom{n-m+1}{2}.
\]
If $m<n/2$, then $m+2\le n-m+1$, so the right-hand side is nonpositive. If $m>n/2$, the right-hand side is nonnegative. Hence $g_n(m)$ decreases up to the midpoint and increases afterwards, and its minimum is attained when $m$ is as close to $n/2$ as possible, namely at
$m\in\set{\floor{n/2},\ceil{n/2}}$.
Substituting these values yields the stated formulas.
\end{proof}

\section{The case involving positive integers}\label{sec:exact-reduction}

Throughout this section we fix
$t\ge 2, n\ge t+1, k=-t$, and $N=n-t\ge 1$.
Let $[1-t,n-t]=Q_t\sqcup P_N$,
where $Q_t=[1-t,0]$ and $P_N=[1,N]$.
We further split the nonpositive part into
$Q_t=R_t\sqcup\{0\}$, where $R_t=[1-t,-1]$.
Thus $R_t$ is the negative block and $Q_t$ is the nonpositive block.

Let $\Delta:[1-t,n-t]\to\{0,1\}$ be any $2$-coloring of $[1-t,n-t]$. We define the following parameters associated with $\Delta$.
\begin{align*}
a&=\abs{\Delta^{-1}(0)\cap R_t}, \qquad q=\abs{\Delta^{-1}(0)\cap P_N},\qquad 
\varepsilon=
\begin{cases}
1,&\Delta(0)=0,\\
0,&\Delta(0)=1.
\end{cases}
\end{align*}
Thus $a$ is the number of negative integers of color $0$, $q$ is the number of positive integers of color $0$, and the parameter $\varepsilon$ records the color of $0$. The ranges are
$0\le a\le t-1,\varepsilon\in\{0,1\}$, and $0\le q\le N$.

The following lemma describes the symmetry under exchanging the two colors, and shows that it suffices to consider colorings with \(q\le N/2\).
\begin{lemma}\label{lem:complement}
If $\overline\Delta=1-\Delta$, then $\Delta$ and $\overline\Delta$ have the same number of monochromatic solutions to $\calE$. Moreover, if $\Delta$ has parameters $(a,\varepsilon,q)$, then $\overline\Delta$ has parameters
\[
(t-1-a,\ 1-\varepsilon,\ N-q).
\]
Consequently, in minimizing the number of monochromatic solutions we may restrict attention to colorings satisfying $q\le \frac N2$.
\end{lemma}

\begin{proof}
A triple is monochromatic under $\Delta$ if and only if it is monochromatic under $\overline\Delta$, because passing from $\Delta$ to $\overline\Delta$ merely interchanges the two colors. The parameter transformation follows directly from the definitions.

If $q>N/2$, then $\overline\Delta$ has $N-q<N/2$ positive integers of color $0$ and yields the same monochromatic solution count as $\Delta$. Therefore, when minimizing, it is enough to consider colorings with $q\le N/2$.
\end{proof}

Because every element of $Q_t$ is nonpositive and every element of $P_N$ is positive, the monotonicity condition $x_1\le x_2\le x_3$ forces every solution to belong to exactly one of the following four types:
\[
PPP,\qquad QQQ,\qquad QQP,\qquad QPP.
\]
Here, for example, $QQP$ means that $x_1,x_2\in Q_t$ and $x_3\in P_N$. We now compute the contributions of the last three types exactly.
Let \(N_{PPP}\), \(N_{QQQ}\), \(N_{QQP}\), and \(N_{QPP}\) denote the numbers of monochromatic solutions of the corresponding types, respectively.
We now compute \(N_{QQQ}\), \(N_{QQP}\), and \(N_{QPP}\) exactly.

We write
$s_0=a+\varepsilon$ and
$s_1=t-s_0=t-a-\varepsilon$.
Hence $s_0$ and $s_1$ are the numbers of nonpositive integers of colors $0$ and $1$, respectively.
\begin{lemma}\label{lem:QQQ}
For every coloring $\Delta$ with parameters $(a,\varepsilon,q)$,
\[
N_{QQQ}=\binom{s_0+2}{3}+\binom{s_1+2}{3}-1.
\]
\end{lemma}

\begin{proof}
All three variables lie in $Q_t=[1-t,0]$. As in the proof of Theorem~\ref{thm:nonpositive}, every nondecreasing triple in $Q_t$ satisfies $\calE$, except for $(0,0,0)$.

The part of $Q_t$ with color $0$ contains $s_0$ elements, so it contributes $\binom{s_0+2}{3}$ nondecreasing monochromatic triples. Similarly, the part with color $1$ contributes $\binom{s_1+2}{3}$ such triples. Exactly one of the two colors contains $0$, so the exceptional triple $(0,0,0)$ has been counted once and must be removed. This yields the formula.
\end{proof}

\begin{lemma}\label{lem:QQP}
For every coloring $\Delta$ with parameters $(a,\varepsilon,q)$,
\[
N_{QQP}=\binom{s_0+1}{2}\,q+\binom{s_1+1}{2}\,(N-q).
\]
\end{lemma}

\begin{proof}
Fix a color $c\in\{0,1\}$. The number of nondecreasing pairs $(u,v)$ of color $c$ inside $Q_t$ equals $\binom{s_c+1}{2}$, because choosing a nondecreasing pair from $s_c$ ordered values is the same as choosing a multiset of size $2$.
Let $x$ be any positive integer of the same color $c$. Since $u,v\in Q_t$, we have
$u\le v\le 0<x$,
and therefore
$u+v\le 0<x$.
Thus every triple $(u,v,x)$ is a solution to $\calE$.
For color $0$, there are $q$ choices for $x$ and $\binom{s_0+1}{2}$ choices for $(u,v)$. For color $1$, there are $N-q$ choices for $x$ and $\binom{s_1+1}{2}$ choices for $(u,v)$. Summing the two colors proves the claim.
\end{proof}

\begin{lemma}\label{lem:QPP}
For every coloring $\Delta$ with parameters $(a,\varepsilon,q)$,
\[
N_{QPP}=
a\binom{q+1}{2}
+(t-1-a)\binom{N-q+1}{2}
+\varepsilon\binom{q}{2}
+(1-\varepsilon)\binom{N-q}{2}.
\]
\end{lemma}

\begin{proof}
We separate the negative first coordinates from the first coordinate equal to $0$.

Let $x_1<0$ be of color $0$. For any  pair $(x_2,x_3)$ with color $0$ satisfying
$x_2\le x_3, x_2,x_3\in P_N$,
we have
$x_1+x_2<x_2\le x_3$,
so $(x_1,x_2,x_3)$ is a solution to $\calE$. The number of nondecreasing pairs with color $0$ in $P_N$ is $\binom{q+1}{2}$. Since there are exactly $a$ negative integers of color $0$, the total contribution from negative first coordinates with color $0$ is
$a\binom{q+1}{2}$.
Similarly, the contribution from negative first coordinates with color $1$ is
$(t-1-a)\binom{N-q+1}{2}$.

Then, let $x_1=0$.
If $\varepsilon=1$, then $0$ has color $0$. A triple $(0,x_2,x_3)$ with $x_2,x_3\in P_N$ and both of color $0$ is a solution to $\calE$ if and only if
$0+x_2<x_3$,
that is, if and only if $x_2<x_3$. The number of strictly increasing ordered pairs chosen from $q$ positive integers is $\binom{q}{2}$. Thus the contribution from $0$ is $\varepsilon\binom{q}{2}$.
Likewise, if $0$ has color $1$, then the contribution is $(1-\varepsilon)\binom{N-q}{2}$.

Adding the two parts completes the proof.
\end{proof}

In the end, we show the exact reduction theorem.

\begin{theorem}\label{thm:exact-reduction}
Let $t\ge 2$, let $n\ge t+1$, and put $N=n-t$. Then
\[
M_{-t}(n)=
\min_{\varepsilon\in\{0,1\}}
\ \min_{0\le a\le t-1}
\ \min_{0\le q\le \floor{N/2}}
G_{t,N,a,\varepsilon}(q),
\]
where
\begin{align*}
G_{t,N,a,\varepsilon}(q)
={}&
S(1,q)+S(q+1,N) \\
&+\binom{a+\varepsilon+1}{2}\,q
+\binom{t-a-\varepsilon+1}{2}\,(N-q) \\
&+a\binom{q+1}{2}
+(t-1-a)\binom{N-q+1}{2} \\
&+\varepsilon\binom{q}{2}
+(1-\varepsilon)\binom{N-q}{2} \\
&+\binom{a+\varepsilon+2}{3}
+\binom{t-a-\varepsilon+2}{3}-1.
\end{align*}
Moreover, for every fixed triple $(a,\varepsilon,q)$ with $0\le q\le N/2$, equality is attained by a coloring whose positive part is the two-block coloring
\[
[1,q]\text{ in color }0,
\qquad
[q+1,N]\text{ in color }1.
\]
\end{theorem}

\begin{proof}
By Lemma~\ref{lem:complement}, there exists an optimal coloring with $q\le N/2$. Fix such a coloring $\Delta$ and let $(a,\varepsilon,q)$ be its parameters.
We decompose the monochromatic solutions counted by $s(\Delta;1-t,n-t)$ into the four types $PPP$, $QQQ$, $QQP$, and $QPP$.

We now consider Type $PPP$.
All three variables lie in the positive interval $P_N=[1,N]$. Inside $P_N$, the coloring $\Delta$ has exactly $q$ integers of color $0$ and $N-q$ integers of color $1$. Since $q\le N/2$, Theorem~\ref{thm1:kosek} applies and yields
$N_{PPP}\ge S(1,q)+S(q+1,N)$.

Equality holds exactly when the restriction of $\Delta$ to $P_N$ is the two-block coloring
\[
[1,q]\text{ in color }0,
\qquad
[q+1,N]\text{ in color }1.
\]

Then we consider Types $QQQ$, $QQP$, and $QPP$.
By Lemmas~\ref{lem:QQQ}, \ref{lem:QQP}, and \ref{lem:QPP}, these three contributions are
\begin{align*}
N_{QQQ}&=\binom{s_0+2}{3}+\binom{s_1+2}{3}-1,\\
N_{QQP}&=\binom{s_0+1}{2}\,q+\binom{s_1+1}{2}\,(N-q),\\
N_{QPP}&=a\binom{q+1}{2}+(t-1-a)\binom{N-q+1}{2}+\varepsilon\binom{q}{2}+(1-\varepsilon)\binom{N-q}{2},
\end{align*}
where $s_0=a+\varepsilon$ and $s_1=t-a-\varepsilon$.

Adding the four types gives
$s(\Delta;1-t,n-t)\ge G_{t,N,a,\varepsilon}(q)$.
Since this is true for every optimal coloring with $q\le N/2$, we obtain the lower bound
\[
M_{-t}(n)\ge
\min_{\varepsilon\in\{0,1\}}
\ \min_{0\le a\le t-1}
\ \min_{0\le q\le \floor{N/2}}
G_{t,N,a,\varepsilon}(q).
\]

To prove the reverse inequality, fix arbitrary parameters
\[
\varepsilon\in\{0,1\},\qquad 0\le a\le t-1,\qquad 0\le q\le \floor{N/2}.
\]
Construct a coloring $\Delta$ as follows:
\begin{itemize}[leftmargin=2em]
\item color exactly $a$ of the $t-1$ negative integers in $R_t=[1-t,-1]$ by $0$ and the remaining $t-1-a$ by $1$;
\item color $0$ by $\varepsilon$;
\item color the positive interval $P_N=[1,N]$ by the two-block coloring
\[
[1,q]\text{ in color }0,
\qquad
[q+1,N]\text{ in color }1.
\]
\end{itemize}
For this coloring, the $PPP$ contribution is exactly $S(1,q)+S(q+1,N)$ by Theorem~\ref{thm1:kosek}, and the other three contributions are exactly the quantities in Lemmas~\ref{lem:QQQ}, \ref{lem:QQP}, and \ref{lem:QPP}. Hence
\[
s(\Delta;1-t,n-t)=G_{t,N,a,\varepsilon}(q).
\]
Since the parameters were arbitrary, we obtain the reverse inequality
\[
M_{-t}(n)\le
\min_{\varepsilon\in\{0,1\}}
\ \min_{0\le a\le t-1}
\ \min_{0\le q\le \floor{N/2}}
G_{t,N,a,\varepsilon}(q).
\]
Combining the two inequalities completes the proof.
\end{proof}

\section{Cubic asymptotics}\label{sec:asymptotic}

We now derive the cubic main term from the exact reduction theorem.
We write \(F(N)=O_t(N^\alpha)\) if there exists a constant \(C_t>0\), depending only on \(t\), such that
$|F(N)|\le C_t N^\alpha$
for all sufficiently large \(N\); and we write \(F(N)=o_t(N^\alpha)\) if
$\lim_{N\to\infty}\frac{F(N)}{N^\alpha}=0$
for each fixed \(t\).

\begin{theorem}\label{thm:asymptotic}
For every fixed integer $k\le -2$,
\[
M_k(n)=\frac{n^3}{12(1+2\sqrt2)^2}(1+o_k(1)).
\]
\end{theorem}

\begin{proof}
Fix $t=-k\ge 2$ and put
$N=n-t$.
For $n\ge t+1$, Theorem~\ref{thm:exact-reduction} gives
\[
M_{-t}(n)=
\min_{\varepsilon\in\{0,1\}}
\ \min_{0\le a\le t-1}
\ \min_{0\le q\le \floor{N/2}}
G_{t,N,a,\varepsilon}(q),
\]
where
\begin{align*}
G_{t,N,a,\varepsilon}(q)
={}&
S(1,q)+S(q+1,N) \\
&+\binom{a+\varepsilon+1}{2}\,q
+\binom{t-a-\varepsilon+1}{2}\,(N-q) \\
&+a\binom{q+1}{2}
+(t-1-a)\binom{N-q+1}{2} \\
&+\varepsilon\binom{q}{2}
+(1-\varepsilon)\binom{N-q}{2} \\
&+\binom{a+\varepsilon+2}{3}
+\binom{t-a-\varepsilon+2}{3}-1.
\end{align*}
We now simplify this formula to its cubic main term.

Since $t$ is fixed, the parameters satisfy
$0\le a\le t-1,
\varepsilon\in\{0,1\}$, and
$0\le q\le \frac N2$.
Hence the coefficients
\[
\binom{a+\varepsilon+1}{2},\quad
\binom{t-a-\varepsilon+1}{2},\quad
a,\quad t-1-a,\quad \varepsilon,\quad 1-\varepsilon
\]
are all bounded in terms of $t$ alone. Therefore the linear terms in $q$ and $N-q$ are $O_t(N)$, the binomial terms of order two are $O_t(N^2)$, and the constant term is $O_t(1)$. Uniformly for all admissible $(a,\varepsilon,q)$,
\begin{equation}\label{eq:G-first-reduction}
G_{t,N,a,\varepsilon}(q)=S(1,q)+S(q+1,N)+O_t(N^2).
\end{equation}
Therefore, 
\[
M_{-t}(n)=
\min_{\varepsilon\in\{0,1\}}
\ \min_{0\le a\le t-1}
\ \min_{0\le q\le \floor{N/2}}
G_{t,N,a,\varepsilon}(q)=\min_{0\le q\le \floor{N/2}} (S(1,q)+S(q+1,N)+O_t(N^2)).
\]
From Theorem \ref{thm7:KRSS}, we have
\[
M_{-t}(n)=\frac{N^3}{12(1+2\sqrt2)^2}(1+o_t(1)).
\]
Since $t$ is fixed, it follows that 
$N=n-t=n+O_t(1)$,
and hence
$N^3=n^3+O_t(n^2)$.
Therefore
\[
M_{-t}(n)=\frac{n^3}{12(1+2\sqrt2)^2}(1+o_t(1)).
\]
Replacing $t$ by $-k$ gives the equivalent formulation for fixed $k\le -2$.
\end{proof}

\section{Concluding remarks}

Theorem~\ref{thm:exact-reduction} shows that the difficulty of the case $k\le -2$ is localized in the fixed nonpositive block. Once the colors on $[1-t,0]$ are encoded by the finite parameters
\[
a=\abs{\Delta^{-1}(0)\cap[1-t,-1]}
\qquad\text{and}\qquad
\varepsilon=\begin{cases}1,&\Delta(0)=0,\\0,&\Delta(0)=1,\end{cases}
\]
the entire optimization problem is reduced to a single integer parameter $q$, the number of positive integers of color $0$. For each triple $(a,\varepsilon,q)$, the positive part of an optimal coloring is exactly the KRSS two-block coloring.

In particular, the open problem left in \cite{KRSS} for $k\le -2$ admits an exact finite reduction and the same cubic asymptotic constant as in the nonnegative case. It would be interesting to know whether a similar phenomenon holds for more general systems of linear inequalities in which a fixed nonpositive block interacts with a long positive tail.

\end{document}